\documentclass[11pt,twoside]{article}
\usepackage{latexsym}
\usepackage{amssymb,amsbsy,amsmath,amsfonts,amssymb,amscd,stmaryrd}
\usepackage{epsfig, graphicx,mathrsfs}
\usepackage{color}
\newcommand{\commentout}[1]{}
\newcommand{\R}{\mathbb{R}}
\newcommand{\N}{\mathbb{N}}

\newcommand {\ffi}  {\varphi}

\newcommand {\Chi} {{\bf \raise 2pt \hbox{$\chi$}} }
\newcommand{\dd}{\mathrm{d}}
\renewcommand{\L}{\textnormal{L}}

\newcommand{\D}{\textnormal{D}}
\renewcommand{\H}{\textnormal{H}}
\newcommand{\W}{\textnormal{W}}

\newcommand{\beq}{\begin{equation}}
\newcommand{\eeq}{\end{equation}}
\newcommand{\bea} {\begin{array}{rl}}
\newcommand{\eea} {\end{array}}
\newcommand{\bepa}{\left\{ \begin{array}{l}}
\newcommand{\eepa} {\end{array}\right.}
\newtheorem{theorem}{Theorem}[section]
\newtheorem{lemma}[theorem]{Lemma}
\newtheorem{definition}[theorem]{Definition}
\newtheorem{remark}[theorem]{Remark}
\newtheorem{proposition}[theorem]{Proposition}

 \newcommand{\qed}{{ \hfill
                        {\unskip\kern 6pt\penalty 500
                        \raise -2pt\hbox{\vrule\vbox to 6pt{\hrule width 6pt
                        \vfill\hrule}\vrule} \par}   }}
\title{Entropic structure and duality for multiple species cross-diffusion systems}
                       \author{Thomas Lepoutre\footnote{INRIA and Universit\'e Claude Bernard Lyon 1, CNRS UMR 5208, Institut Camille Jordan, 43 blvd. du 11 novembre 1918, F-69622 Villeurbanne cedex, France} \,and Ayman Moussa\footnote{Sorbonne Universit\'es, UPMC Univ Paris 06 \& CNRS, UMR 7598 LJLL, Paris, F-75005, France}}
\begin{document}
\maketitle
\abstract{This paper deals with the existence of global weak solutions for a wide class of (multiple species) cross-diffusions systems. The existence is based on two different ingredients: an entropy estimate giving some gradient control and a duality estimate that gives naturally $\L^2$ control. The heart of our proof is a semi-implicit scheme tailored for cross-diffusion systems firstly defined in \cite{dlmt} and a (nonlinear Aubin-Lions type) compactness result developped in \cite{mou,ACM} that turns the (potentially weak) gradient estimates into almost everywhere convergence. We apply our results to models having an entropy relying on the \emph{detailed balance condition} exhibited by Chen \emph{et. al.} in \cite{chen2016global}.}
\section{Introduction}
In 1979, Shigesada, Kawasaki and Terramoto introduced in \cite{Shigesada1979} the following system (that we denote SKT), on $Q_T:=[0,T]\times\Omega$ where $\Omega\subset\R^d$ is some regular bounded open set
\begin{align*}
\left\{
\begin{array}{l}
\partial_t u_1-\Delta \Big[(d_1+a_{11}u_1+a_{12}u_2)u_1\Big]= u_1(\rho_1-s_{11}u_1-s_{12}u_2),\\
\\
\partial_t u_2-\Delta \Big[(d_2+a_{22}u_2+a_{21}u_1)u_2\Big]= u_2(\rho_2-s_{21}u_1-s_{12}u_2).
\end{array}
\right.
\end{align*}
The latter aims at describing the behavior of two populations (through their density functions $u_1,u_2\geq 0$) involving different mechanisms: self-diffusion ($a_{11}$, $a_{22}$ terms), cross-diffusion ($a_{12},a_{21}$ terms) and growth terms, modelling reproduction ($\rho_1,\rho_2$ terms) or  competition ($s_{ij}$ terms).
The existence theory for the corresponding Cauchy boundary value problem is a rich saga. As far as classical solutions are concerned, the cornerstone of the theory is Amman theorem \cite{amann88,amann90,amann90b} which ensures local existence of solutions and gives also a criterion to check for possible global solution (which amounts to control some Sobolev norms of the solution). As a matter of fact, up to now classical global solutions are only known to exist under strong assumptions on the coefficients, for instance in the case of a weak coupling in the diffusion matrix ($a_{21}=0$ above) known as the triangular case (see \cite{HoanNguPha} for a recent result) or equal diffusion rates without self-diffusion, see \cite{kim1984smooth,lou2015global}.  Concerning weak solutions, it is a striking fact that until the beginning of the 2000's, no global solutions were known for the (full) SKT model. The breakthrough occured in several steps by  J\"ungel \emph{et. al.} in \cite{Galiano_Num_Math,Chen2004,Chen2006}, the core of the construction being the discovery of the following entropy for the system 
$$
\mathcal{E}(t):=\int_\Omega a_{21}(u_1\log u_1-u_1+1)+a_{12}(u_2\log u_2-u_2+1).
$$
As detailed in \cite{Chen2006}, this convex functional is controlled over the time through
\begin{align*}
\mathcal{E}'(t)  + \mathcal{D}(t) \leq C(1+\mathcal{E}(t)),
\end{align*}
where $\mathcal{D}(t)$ is some dissipative (nonnegative) term. The entropy control leads to $u_i\log u_i \in\L^\infty(0,T;\L^1(\Omega))$ and one can also recover (at least) from the dissipative term that $\nabla\sqrt{u_i} \in\L^2(Q_T)$. The previous \emph{a priori} estimates pave the way to the (strong) convergence of adequate approximating procedures. This is not specific to the SKT model and has been transposed since $2006$ to several variations of this system, see subsection \ref{subsec:ent} for more details.

\vspace{2mm}

This manuscript is devoted to an existence theorem for generalizations of the SKT model which take the form
\begin{align}\label{eq:cross}
\partial_t u_i-\Delta (p_i(U)u_i)=r_i(U)u_i,
\end{align}
where $1\leq i\leq I$, $U=(u_i)_{1\leq i\leq I}$,
\begin{equation}\label{eq:p_i_reg}
p_i\in \mathscr{C}^0(\R_+^I,\R_+)\cap\mathscr{C}^1((\R_+^*)^I,\R_+),
\end{equation}
 and the reaction terms are continuous functions on $\R^I_+$ that can have for instance form
\begin{equation}\label{eq:reac_form}
r_i(U)=\rho_i-\sum_{j=1}^I c_{ij}u_j^{\alpha_{ij}},\quad \rho_i,c_{ij},\alpha_{ij}\geq 0,\quad  \alpha_{ij}<1.
\end{equation}
Introducing $A(U):=(p_i(U)u_i)_{1\leq i\leq I}$ and $R(U) := (r_i(U)u_i)_{1\leq i\leq I}$, the set of scalar equations \eqref{eq:cross} is equivalent to the vectorial one 
\begin{align}
\label{eq:crossvect}
\partial_t U - \Delta \big[A(U)\big] = R(U).
\end{align}
We will use both formulations in the sequel, keeping in mind that capital letters refer to vectors and lowercase to scalars. 

\textbf{Notations: } In all what follows, we denote by $Q_T=(0,T)\times \Omega$ the parabolic cylinder. The space $\H^{-1}(\Omega)$ is the dual of the elements of $\H^1(\Omega)$ having $0$ average, an element $u\in\H^{-1}(\Omega)$ is thus characterized by 
\begin{align*}
\forall \ffi\in\H^1(\Omega),\quad \int_\Omega (\varphi-\overline{\varphi})u\leq \|\varphi\|_{\H^{1}(\Omega)}\|u\|_{\H^{-1}(\Omega)},
\end{align*}
where $\overline{\ffi}$ is the average of $\ffi$ on $\Omega$. For two vectors $X=(x_i)_i$ and $Y=(y_i)_i$ of $\R^I$ we write $X\leq Y$ if and only if the inequality is satisfied for each of their components, and use the same convention for $<$. The tensor $X:Y$ is the square-matrix $(x_i y_j)_{i,j}$. Finally, if $U$ is a vector-valued or matrix-valued function defined on $Q_T$ and $\textnormal{E}$ is some vector space of (scalar) functions defined on $Q_T$, we write simply $U\in \textnormal{E}$ to specify that each components of $U$ belongs to $\textnormal{E}$.
\subsection{Weak solutions and entropy structure}\label{subsec:ent}
Let's recall how can one recover the entropy structure exhibited in \cite{Chen2006} when dealing with a general system as \eqref{eq:crossvect}. Taking formally the inner product of \eqref{eq:crossvect} by $\nabla \mathcal{H}(U)$ for some function $\mathcal{H}:(\R_+^*)^I \rightarrow \R$ we get by a standard computation (with the usual repeated index convention)
\begin{align*}
\frac{\dd}{\dd t}\int_\Omega \mathcal{H}(U)  + \int_\Omega \langle \partial_j U,\D^2(\mathcal{H})(U) \D(A)(U) \partial_j U\rangle = \int_\Omega \nabla \mathcal{H}(U) \cdot R(U),
\end{align*}
where $\D(A)$ and $\D^2(\mathcal{H})$ are respectively the jacobian and hessian matrix of $A$ and $\mathcal{H}$. With a slight abuse of notations we rephrase the previous identity: 
\begin{align*}
\frac{\dd}{\dd t}\int_\Omega \mathcal{H}(U)  + \int_\Omega \langle \nabla U, \D^2(\mathcal{H})(U)\D(A)(U) \nabla U\rangle = \int_\Omega \nabla \mathcal{H}(U) \cdot R(U),
\end{align*}
which justifies the following definition 
\begin{definition}[Entropy]
\label{def:ent}We say that the system has an \emph{entropy function}, if there exists a function $\mathcal{H}:(\R_+^*)^I\rightarrow \R_+$ which is $\mathscr{C}^2$, convex and such that for any $X=(x_i)_i>0$
\begin{align*}
\D^2(\mathcal{H})(U)\D(A)(X)\text{ is  positive-semidefinite}.
\end{align*}
\end{definition}
In particular if $\mathcal{H}$ is an entropy function satisfying furthermore 
\begin{equation*}
\nabla \mathcal{H}(U)\cdot R(U)\leq C(1+\mathcal{H}(U)),
\end{equation*}
this leads (formally) to the entropy inequality 
\begin{equation}\label{eq:entropy_ineq}
\frac{\dd}{\dd t}\int_\Omega \mathcal{H}(U)+\underbrace{\int_\Omega \langle \nabla U, \D^2(\mathcal{H})(U)\D(A)(U)\nabla U\rangle }_{\geq 0}\leq C\Big(1+\int_\Omega \mathcal{H}(U)\Big),
\end{equation}
where the braced term of the l.h.s. is the generalization of the dissipation term $\mathcal{D}(t)$ of J\"ungel \emph{et. al.} Depending on the situation, \eqref{eq:entropy_ineq} allow an adequate control (in terms of oscillations and concentration) of all the nonlinear terms in order to build weak solutions. After \cite{Chen2006} several works focused on systems exhibiting this type of structure and extracted (gradient-like) information from this entropy estimate \cite{Jungel_Stelzer,jungel2016cross,Zamponi,dlmt,DesLepMou}.  Note also that in some specific cases, boundedness can arise from the entropy control \cite{jungbound}.

\subsection{Duality estimates and cross-diffusion}
 In some cases, one needs an additional estimate to deal with concentration issues (\emph{i.e.} non equi-integrability). In such a situation a powerful tool are the duality estimates introduced by Michel Pierre and Didier Schmitt \cite{PiSc}, see \cite{DesLepMou,dlmt,Lepoutre_JMPA} for examples of use in the context of cross-diffusion models. Let us recall briefly a paradigm of such estimates 
\begin{lemma}\label{lem:dualite}
Fix $\rho\geq 0$ and consider an integrable function $\mu$,  satisfying $\mu\geq \nu>0$ on $Q_T$. Any nonnegative smooth solution $v$ of
\begin{align*}
\partial_t v-\Delta \big[ \mu v \big] & \leq \rho v, \text{ on }Q_T\\
v(0)&=v_0,\\
 \partial_n (\mu v)&=0,\text{on }\partial\Omega,
\end{align*}
satisfies the \emph{a priori} $\L^2(Q_T)$ estimate
$$
\int_{Q_T} \mu v^2\leq e^{2\rho T}\left(\|v_0- \overline{v_0}\|_{\H^{-1}(\Omega)}^2+(\bar v^0)^2\int_{Q_T}\mu\right).
$$
\end{lemma}
\begin{remark}
Many progress have been made recently concerning these estimates. For instance, if $p$ is assumed to be bounded, one can actually recover an estimate in some $\L^p(Q_T)$ space with $p>2$, see \cite{canizo2014improved} for more details. 
\end{remark}
In the case of continuous coefficients $p_i$, applying this result for
\begin{align*}
u:=\sum_{i=1}^I u_i,\quad \mu:=\frac{\sum_{i=1}^I p_iu_i}{\sum_{i=1}^I u_i},
\end{align*}
 leads (more details in the appendix) to the following useful \emph{a priori} bound
\begin{lemma}\label{lem:dualsum}
Assume $r_i(U)\leq \rho$ for any $U\geq 0$. If for all $i$ we have $p_i\in\mathscr{C}^0(\R_+^I)$, then any nonnegative smooth solution $(u_i)_{1\leq i\leq I}$  of \eqref{eq:cross} with initial condition $(u_i^0)_{1\leq i\leq I}\in\H^{-1}(\Omega)\cap\L^1(\Omega)$ and boundary conditions $\partial_n(p_i(U)u_i) = 0$ satisfies 
\begin{align}
\label{ineq:dualsum1}\int_{Q_T}\left(\sum_{i=1}^I u_i\right)\left(\sum_{i=1}^I p_i(U) u_i\right)\leq C,
\end{align}
where the constant $C$ depends only on $\Omega,T,(u_i^0)_{1\leq i\leq I},\rho$ and the functions $p_i$. 
\end{lemma}

A schematic proof of this lemma is postponed to the appendix, we emphasize that the continuity on the whole set $\R_+^I$ of the $p_i$ is critically used here (which excludes, at least at first sight the case studied in \cite{Zamponi}). 

\begin{remark}
In particular, a suitable discrete version of this estimate will hold for the semi-discrete approximation (see estimate \eqref{for38}). This estimate will be of crucial to avoid equi-integrability issues in the final passage to the limit.
\end{remark}
  
\subsection{Statement of the main theorem}
To state our main theorem, we need the introduction of a strengthened entropy notion: 
\begin{definition}[Uniform entropy]
\label{def:ent:uni} An entropy in the sense of Definition \ref{def:ent} is called \emph{uniform}, if there exists continuous functions $f_i:\R_+^* \rightarrow \R_+^*$ such that for all $X=(x_i)_i>0$
\begin{align*}
\D^2(\mathcal{H})(X)\D(A)(X) \geq \textnormal{Diag}(f_i(x_i))^2,
\end{align*}
in the sense of symmetric matrices.
\end{definition}
The proof of the following Theorem relies on the previous entropy and dual estimates together with the approximation procedure introduced in \cite{dlmt} (see Section \ref{sec:scheme}). 
\begin{theorem}\label{thm:main}
  Let $\Omega$ be a smooth domain. Assume the coefficients satisfy: 
\begin{itemize}
\item[\textnormal{\textbf{H1}}] The functions $p_i$ satisfy \eqref{eq:p_i_reg} and the $r_i$ are continuous from $\R_+^I$ to  $\R$.
\item[\textnormal{\textbf{H2}}] For all $i$, $p_i$ is lower bounded by some positive constant $\alpha>0$ and $r_i$ is upper bounded by a positive constant $\rho>0$. 
\item[\textnormal{\textbf{H3}}] $A$ is a homeomorphism from $\R_+^I$ to itself.
\end{itemize}  
  Assume the existence of a uniform entropy function $\mathcal{H}$ (in the sense of Definition \ref{def:ent:uni}) satisfying for some $C >0$ and any $0\leq X=(x_i)_{1\leq i\leq I}$
\begin{align}
\label{ineq:H2} \nabla \mathcal{H}(X) \cdot R(X) &\leq C\left(1+\mathcal{H}(X)\right).
\end{align}
Assume finally that the function $R$ satisfies (for some norm $\|\cdot\|$ on $\R^I$)
\begin{equation}\label{eq:equiint}
\| R(X)\|=\textnormal{o}\left(\left(\sum_{i=1}^I p_i(X)x_i\right)\left(\sum_{i=1}^I x_i\right)+\mathcal{H}(X)\right), \text{ as } \|X\| \rightarrow \infty.
\end{equation}
Then, for any $0\leq U_{\textnormal{in}} \in \L^1(\Omega)\cap \H^{-1}(\Omega)$, such that $ \mathcal{H}(U_{\textnormal{in}})\in\L^1(\Omega)$, there exists $0\leq U \in\L^1(Q_T)$ such that $A(U)\in\L^1(Q_T)$ and $R(U)\in\L^1(Q_T)$ which is a weak solution system \eqref{eq:crossvect} with initial data $U_{\textnormal{in}}$ and homogeneous Neuman boundary conditions, \emph{i.e.} for  all $\Psi\in\mathscr{C}^1_c([0,T);\mathscr{C}^2(\overline{\Omega})^I)$ satisfying $\partial_n \Psi = 0$ on $\partial \Omega$, there holds 
 \begin{align*}
- \int_\Omega U_{\textnormal{in}} \cdot\Psi(0,\cdot) = \int_{Q_T} \Big(A(U) \cdot \Delta \Psi + R(U) \cdot \Psi\Big).
\end{align*}
Moreover, this solution satisfies the following estimate on $[0,T]$:
\begin{align}
\int_\Omega\mathcal{H}(U(t)) +\int_0^t\int_{\Omega} \langle \nabla U,\D^2(\mathcal{H})(U)\D(A)(U)\nabla U\rangle  &\leq (1+e^{2CT})\left(1+\int_\Omega \mathcal{H}(U_{\textnormal{in}})\right),
\label{eq:gradient_control}
\end{align}
where $C$ is the constant introduced in \eqref{ineq:H2}. 
\end{theorem}
\begin{remark}[About the hypothesis]\label{rem:hyp}
We discuss a little bit the hypothesis. 
\begin{itemize}
\item The entropy structure hypothesis is crucial to get compactness estimates on the gradients. Notice that in this regard; our assumption of uniform entropy is actually quite weak. For instance it allows cases in which  the dissipation gives bounds on quantities like $e^{-u}|\nabla u|^2$, which falls out of the scope of the standard Aubin-Lions-Simon Lemma.
 \item Continuity of the $p_i$ on the whole set $\R_+^I$ is essential to the derivation of the $\L^2$ estimate (meaning that some counterexampels could be constructed). The lower bound on the $p_i$ might possibly be relaxed but it would require substantially more work.
 \item  The upper bound on $r_i$ could probably be optimized to an upper bound for a convex combination of the $r_i$, typically $\sum_{i=1}^I u_i r_i(U)\leq \rho \sum_{i=1}^I u_i$, which is the control we actually need.
 \item  We chose hypothesis \eqref{eq:equiint} to avoid dimensional consideration and to emphasize that as far as reaction terms are concerned and \eqref{ineq:H2} is satisfied, the only issue concerns equi-integrability estimates. However, this assumption could be weakened if the combination of gradients in the dissipation and other bounds ensures the $\L^1$ equi-integrabilty of $R(U)$.
 \item  Finally hypothesis \textbf{H3} might seem restrictive but we will see in Section \ref{sec:structure} that in many cases of interests, it follows  directly from the entropic structure itself.
\end{itemize}
\end{remark}
\begin{remark}[About the type of solutions]
To emphasize the robustness of the approximation procedure introduced in \cite{dlmt}, we do not derive optimal results in terms of regularity or integrability of $U$, $A(U)$ and $R(U)$. In practice the latter are more than merely integrable and using the dissipation estimate  \eqref{eq:gradient_control}, one can enlarge the space of test functions and replace in the weak formulations the terms $p_i(U)u_i\Delta \psi_i$ by $- \nabla(p_i(U)u_i)\cdot \nabla \psi_i $. 
\end{remark}

\subsection{Application to multiple species with detail balance.}
An important part of the literature is devoted to 2 species system. Concerning models with multiple (more than 2) species, there have been some studies see \cite{2000_article_LouMarNi} for instance. These studies generally do not focus on the Cauchy problem but more on the existence of specific solutions (stationary or  periodic solutions for instance). Examples of global existence of weak solutions for multiple species can be found in \cite{dlmt} (Section 5.2) and more importantly in \cite{chen2016global}. One of the purposes of the paper is to show that  weak solutions can be obtained for systems like \eqref{eq:cross} under mainly structural hypotheses using the methods derived in \cite{dlmt}. We try to get rid as much as possible of hypotheses that arise from technical issues.
Concerning the entropy structure of more than 2 species system, a very recent work \cite{chen2016global}  identifies a structure for system with pressure of the form 
\begin{equation}\label{eq:jung_form}
p_{i}(U)=d_i+\sum_{j=1}^I m_{ij} u_j^s,
\end{equation}
where $m_{ij}\geq 0$ and $s>0$. The focus in \cite{chen2016global} is made on the entropy and all the approximation is based on a suitable construction and extension of entropic variables. The authors exhibited the following detailed balance condition which for a matrix $M=(m_{ij})$  writes
\begin{equation}
\exists \pi\in(\R_+^*)^I\,:\,\quad \forall i\not=j,\quad \pi_i m_{ij}=\pi_jm_{ji}.\label{eq:detail_balance}
\end{equation}
This assumption is not sufficient to ensure the existence of an entropy (and more interstingly global weak solutions) and the authors had to consider moreover either a dimensional constraint on the exponent $s$ or a \emph{self $>$ cross} assumption (that is: the $m_{ii}$ are larger than the $m_{ij}$ in some sense). Schematically their results allow to prove existence of global weak solutions in the following two general cases: 
\begin{itemize}
\item $\max(0,1-\frac{2}{d}) < s$ and assumption of type \emph{self $>$ cross};
\item $\max(0,1-\frac{2}{d}) < s \leq 1$ and detailed balance.
\end{itemize}
Thus, the detailed balance condition can be seen as a way to avoid an assumption of the type \emph{self $>$ cross}, but applies only for sublinear exponents. The (dimension dependent) lower-bound on the exponent is related to the fact that the regularity and the integrability of the solution is derived from the entropy inequality (see end of Section 1 in \cite{chen2016global}) and does not make advantage of the $\L^2$ structure that one exhibits when exploiting Lemma~\ref{lem:dualite}. 
In this manuscript we  identify in Section~\ref{sec:structure} natural extensions (including different exponents $s_i$) of the structure identified in \cite{chen2016global}. More importantly, we show how the use of a different approximation procedure allow to get rid of most dimensional constraint. Typically, we are able to obtain the following (non optimal) improvement of Theorem 1 and Theorem 2 of \cite{chen2016global} 
\begin{theorem}\label{thm:genjun}
Let $\Omega$ be a smooth domain. Assume the $p_i$ have form $p_i(U)=d_i+\sum_{j}m_{ij}u_j^{s_j}$, with $s_j>0$, $s_is_j\leq 1, i\not=j$ and \eqref{eq:detail_balance}. Assume the reaction have the form \eqref{eq:reac_form}.
Finally, assume $0\leq u_i^0\in \L^1(\Omega)\cap \H^{-1}(\Omega)$ and $$\int_\Omega h_i(u_i) <\infty \text{ where } h_i(u_i)=\begin{cases} \dfrac{u_i^{s_i}-s_i u_i+s_i-1}{s_i-1} \text{ if }s_i\not=1\\ u_i\log u_i-u_i+1 \text{ if } s_i=1.\end{cases}$$  Then for there exists $0\leq (u_i)_{1\leq i\leq I} \in\L^1(Q_T)$ such that $p_i(U)u_i\in\L^1(Q_T)$ and $r_i(U)u_i\in\L^1(Q_T)$ which is a weak solution system \eqref{eq:cross} with initial data $(u_i^0)_{1\leq i\leq I}$ and homogeneous Neuman boundary conditions, \emph{i.e.} for all $\psi\in\mathscr{C}^1_c([0,T);\mathscr{C}^2(\overline{\Omega})^I)$ satisfying $\partial_n \psi = 0$ on $\partial \Omega$, there holds 
 for all $1\leq i \leq I$
 \begin{align*}
- \int_\Omega u_i^0\psi(0,\cdot) = \int_{Q_T} \Big(p_i(U)u_i \Delta \psi + r_i(U)u_i\psi \Big).
\end{align*} 
Moreover, we have for some constant $C=C(T,\Omega,(u_i^0)_{1\leq i \leq I})$
$$\sum_{i=1}^I \int_\Omega h_i(u_i)+\sum_{i=1}^I \int_{Q_T} u_i^{s_i-2}|\nabla u_i|^2\leq C.$$
\end{theorem}
\begin{remark}
Note that this statement does not depend on the dimension. In comparison with the results of \cite{chen2016global}, this Theorem is weaker on only one point: the constraint $s_i s_j\leq 1$ forbids the case $s_i = s_j =s >1$ considered in \cite{chen2016global}. We could of course include this type of situations in our result but under the price of a \emph{self $>$ cross} assumption. Nevertheless our result includes the scenario of one superlinear exponent $s_i>1$ if it is ``compensated'' by sublinear ones satisfying $s_j \leq s_i^{-1}$, all of this without any \emph{self $>$ cross} assumption on the coefficients. In fact, the superlinear case of \cite{chen2016global} is covered by theorem~\ref{thm:main}. A larger class is explored in Section~\ref{sec:structure}.
\end{remark}

The manuscript is divided as follows. The first two sections are devoted to the proof of Theorem \ref{thm:main}. In Section~\ref{sec:scheme}, we recall the semi-discrete approximation derived in \cite{dlmt} and apply it to our framework. We show how the \emph{a priori} entropy estimate is propagated to the semi-discrete approximation. Section~\ref{sec:aubin} focuses on the asymptotic of the sequence of approximations. We invoke a (nonlinear Aubin-Lions type) compactness result developped in \cite{mou,ACM} to obtain strong convergence and handle the concentration issues in order to conclude that the cluster point is a weak solution of our system, ending the proof of Theorem \ref{thm:main}. Finally in Section~\ref{sec:structure}, we show a general framework for multiple species for which Theorem \ref{thm:main} applies, including the case of Theorem \ref{thm:genjun}.

\section{Scheme} \label{sec:scheme}
\subsection{Semi discretization}
We follow the approximation procedure introduced in \cite{dlmt} for generic systems 
\begin{align*}
\partial_t U - \Delta[A(U)] &= R(U),\\
\partial_n A(U) &= 0,\\
U(0) &= U_{\textnormal{in}},
\end{align*}
where $A:\R_+^I\rightarrow\R_+^I$ and $R:\R_+^I \rightarrow \R^I$ take the following form $A(U) = (p_i(U)u_i)_{i}$, $R(U)=(r_i(U)u_i)_{i}$. The approximation procedure is based on the following semi-implicit scheme 
\begin{align}
\label{eq:scheme1} \frac{U^k-U^{k-1}}{\tau} -\Delta[A(U^k)] &= R(U^k) \text{ on }\Omega,\\
\label{eq:scheme2} \partial_n A(U^k) &=0 \text{ on }\partial\Omega,
\end{align}
where $1\leq k\leq N$ initialized by $U^0=U^0_N$, which is an approximation of $U_{\textnormal{in}}$ such that $U^0_N\in\mathscr{C}^0(\overline{\Omega})$ is smooth, $U^0_N\geq 1/N$ and 
\begin{align*}
\|U^0_N\|_1 &\leq \|U_{\textnormal{in}}\|_1, \\
\|U^0_N\|_{\H^{-1}(\Omega)} &\leq \|U_{\textnormal{in}}\|_{\H^{-1}(\Omega)}, \\
\|\mathcal{H}(U^0_N)\|_1 &\leq \|\mathcal{H}(U_{\textnormal{in}})\|_1. \\
\end{align*}
Equations \eqref{eq:scheme1} - \eqref{eq:scheme2} have to be understood in the following sense 
\begin{definition}\label{def:solscheme}
 Let $\tau>0$ and $0\leq U^{k-1}\in \L^\infty(\Omega)$. We say that a nonnegative vector-valued function $U^k$ is a solution of \eqref{eq:scheme1} - \eqref{eq:scheme2} if $U^k$ lies in $\L^\infty(\Omega)$, $A(U^k)$ lies in $\H^2_{\nu}(\Omega)$ and \eqref{eq:scheme1} is satisfied almost everywhere on $\Omega$.
\end{definition}

Apart from the question of convergence to a global weak solution, the very existence of the sequence $(U^k)_{1\leq k\leq N}$ is nontrivial, because \eqref{eq:scheme1} is highly nonlinear. This issue is solved in \cite{dlmt} precisely under the assumptions \textbf{H1} - \textbf{H2} - \textbf{H3} of Therorem \ref{thm:main}. Under these assumptions, we have the following result (see Theorem 2.2 of \cite{dlmt}):
\begin{theorem}\label{th:approxold}
Assume that \textnormal{\textbf{H1}}, \textnormal{\textbf{H2}}, \textnormal{\textbf{H3}} hold. Let $\Omega$ be a bounded open set of $\R^d$ with smooth boundary.  Fix $T>0$ and an integer $N$ large enough such that $\rho\tau<1/2$, where $\tau:=T/N$ and $\rho$ is the positive number defined in \textnormal{\textbf{H2}}. Fix $\eta>0$ and a vector-valued function  $\L^\infty(\Omega)\ni U^0\geq \eta$. Then there exists a sequence of positive vector-valued functions $(U^k)_{1\leq k\leq N-1}$ in $\L^\infty(\Omega)$ which solves \eqref{eq:scheme1} -- \eqref{eq:scheme2} (in the sense of Definition \ref{def:solscheme}). Furthermore, for all $k\geq 1$ and  $p\in[1,\infty[$, it satisfies the following estimates:
\begin{align}\label{es:depend_on_tau1}
U^k &\in\mathscr{C}^0(\overline{\Omega}), \\
\label{es:depend_on_tau2}U^k &\geq \eta_{A,R,\tau}  \text{ on }\overline{\Omega},\\
\label{es:depend_on_tau3} A(U^k) &\in \W^{2,p}_\nu(\Omega),
\end{align}
where $\eta_{A,R,\tau}>0$ is a positive constant depending on the maps $A$ and $R$ and $\tau$, and
\begin{align}
\label{ineq:l1} \max_{0\leq k\leq N-1}\int_\Omega U^k  &\leq  2^{2\rho \tau N} \int_\Omega U^0, \\
\label{ineq:l1bis} \sum_{k=1}^{N-1} \tau \int_\Omega \left(\rho U^k - R(U^k)\right) &\leq 2^{2\rho \tau N} \int_\Omega U^0,\\
\label{for38}
\sum_{k=0}^{N-1} \tau \int_\Omega \left(\sum_{i=1}^I u_i^k\right)\left(\sum_{i=1}^I A(U^k)_i\right) &\leq  C(\Omega,U^0,A, \rho, N\tau),
\end{align}
where $C(\Omega,U^0,A, \rho, N\tau)$ is a positive depending only on $\Omega$, $A$, $\rho$, $N\tau$ and $\|U^0\|_{\L^1\cap\H^{-1}(\Omega)}$.
\end{theorem}

\subsection{Entropy estimate}\label{sec:entropy}

We try to keep the derivation of the entropy estimate as general as possible. The derivation of the entropy estimate is based on the fact that the approximation are implicit. Indeed, since $U^{k}$ is smooth enough and positive, it is legitimate to multiply system \eqref{eq:scheme1} by $\nabla \mathcal{H}(U^k)$ integrate on $\Omega$ and integrate by parts the gradient term, leading to the equality 
$$
\int_\Omega \nabla \mathcal{H}(U^k)\cdot(U^{k}-U^{k-1})+\tau\int_\Omega \langle \nabla U^k,\D^2(\mathcal{H})(U^k)\D(A)(U^k),\nabla U^k\rangle =\tau\nabla \mathcal{H}(U^k)\cdot R(U^k).
$$
Using \eqref{ineq:H2} we have
$$
\int_\Omega \nabla \mathcal{H}(U^k)\cdot(U^{k}-U^{k-1})+\tau \int_\Omega \langle \nabla U^k,\D^2(\mathcal{H})(U^k)\D(A)(U^k),\nabla U^k\rangle \leq C\tau(1+\mathcal{H}(U^k)).
$$
Since $\mathcal{H}$ is convex, we have
$$
\int_\Omega \nabla \mathcal{H}(U^k)\cdot(U^{k}-U^{k-1})\geq \int_\Omega \mathcal{H}(U^k)-\mathcal{H}(U^{k-1}).
$$
so that we eventually get the discrete analogous of \eqref{eq:entropy_ineq}, that is 
\begin{align}
\label{ineq:entropdis}\int_\Omega \mathcal{H}(U^k)-\int_\Omega \mathcal{H}(U^{k-1})+\tau
\int_\Omega \langle \nabla U^k,\D^2(\mathcal{H})(U^k)\D(A)(U^k),\nabla U^k\rangle\leq C\tau\left(1+\int_\Omega \mathcal{H}(U^k)\right).
\end{align}
For $\tau$ small enough, we have $C\tau<1/2$ and since $T=N\tau$, we thus infer from a discrete-type Gronwall Lemma (see for instance Lemma 3.6 of \cite{DesLepMou}) that, for all $k\in\{1,\cdots,N\}$, 
\begin{align*}
\int_\Omega \mathcal{H}(U^k) &\leq e^{2CT}\left(1+\int_\Omega \mathcal{H}(U^0_N)\right)\\
&\leq e^{2CT}\left(1+\int_\Omega \mathcal{H}(U_{\textnormal{in}})\right),
\end{align*}
from which we eventually deduce, for any $\ell\in\{1,\cdots,N\}$, after summation of \eqref{ineq:entropdis} over $1\leq k\leq \ell$, 
\begin{align}
\label{ineq:entropdisfinal}\int_\Omega \mathcal{H}(U^\ell)+\sum_{k=1}^\ell \tau\int_\Omega \langle \nabla U^k,\D^2(\mathcal{H})(U^k)\D(A)(U^k),\nabla U^k\rangle\leq \left(1+e^{2CT} \right)\left(CT+\int_\Omega \mathcal{H}(U_{\textnormal{in}})\right).
\end{align}

\section{Passing to the limit} \label{sec:aubin}
\subsection{Compactness}
In this section we are going to pass rigorously to the limit $N\rightarrow + \infty$, $\tau\rightarrow 0$ in system \eqref{eq:scheme1} - \eqref{eq:scheme2} that we rewrite here, keeping the track of its dependence w.r.t. $N$ : 
\begin{align*}
\frac{U^k_N-U^{k-1}_N}{\tau} -\Delta[A(U^k_N)] &= R(U^k_N) \text{ on }\Omega,\\
\partial_n A(U^k_N) =0 \text{ on }\partial\Omega. 
\end{align*}
It is convenient to rephrase the previous equation in terms of a continuous space variable, introducing for all $N\geq 1$ the step-in-time function
\begin{align*}
\underline{U}^N := \sum_{k=0}^{N-1} U_N^k \mathbf{1}_{(k\tau,(k+1)\tau]}(t),
\end{align*}
and extend this function by $0$ for negative times, defining $\underline{u}_i^N$ accordingly so that $\underline{U}^N = (\underline{u}_i^N)_i$. The previous set of (discrete in space) equations is actually equivalent to the single one
\begin{align*}
\partial_t \underline{U}^N  = \sum_{k=1}^{N-1} \tau((\Delta[A(U^k)]+R(U^k))\otimes \delta_{t^k} + U^0_N\otimes\delta_0 \in\mathscr{D}'((-\infty,T)\times\Omega),
\end{align*}
where $\delta_{t^k}$ is the Dirac mass at $t^k$. From estimate \eqref{for38} we get the following bounds
\begin{align}
\label{est:comp2}\Big[\underline{U}^N : A(\underline{U}^N)\Big]_N & \text{ bounded  in } \L^1(Q_T), 
\end{align}
whereas estimate \eqref{ineq:entropdisfinal} rephrases, for all $t\in[0,T]$,
\begin{multline}
\label{eq:gradient_control_dis}  \int_\Omega\mathcal{H}(\underline{U}^N(t))  + \int_{0}^t\int_\Omega \langle \nabla \underline{U}^N, \D^2 (\mathcal{H})(\underline{U}^N)\D(A)(\underline{U}^N) \nabla \underline{U}^N\rangle\\ \leq (1+e^{2CT})\left(CT+\int_\Omega \mathcal{H}(U_{\textnormal{in}})\right). 
\end{multline}
We infer in particular from the previous inequality the following bound
\begin{align}
\label{est:comp3} \Big[\langle \nabla \underline{U}^N, \D^2 (\mathcal{H})(\underline{U}^N)\D(A)(\underline{U}^N) \nabla \underline{U}^N\rangle\Big]_N & \text{ bounded in } \L^1(Q_T).
\end{align}
Note that because of Assumption \textbf{H2}, estimate \eqref{est:comp2} leads to the boundedness of $(\underline{U}^N)_N$ in $\L^2(Q_T)$. Up to some subsequence (that we don't label) we can thus assume the existence of $U\in \L^2(Q_T)$  such that 
\begin{align*}
(\underline{U}^N)_N &\operatorname*{\rightharpoonup}_N U,\text{ in } \L^2(Q_T).
\end{align*}
On the other hand, since $\mathcal{H}$ is assumed to be a uniform entropy in the sense of Definition \ref{def:ent}, and since, by \eqref{es:depend_on_tau2}, $\underline{U}^N$ takes positive values on $Q_T$, we infer the existence of continuous functions  $f_i:\R_+^* \rightarrow\R_+^*$  for all $i$   such that 
\begin{align}
\label{est:comp3bis}\langle \nabla \underline{U}^N, \D^2 (\mathcal{H})(\underline{U}^N)\D(A)(\underline{U}^N) \nabla \underline{U}^N\rangle \geq \sum_{i=1}^I f_i(\underline{u}_i^N)^2 |\nabla \underline{u}^N_i|^2.
\end{align}
Now fix $i$ and define $w_N:=\underline{u}_i^N$. Thanks to \eqref{est:comp2}, $(w_N)_N$ is bounded in $\L^2(Q_T)$. Using \eqref{est:comp3} - \eqref{est:comp3bis} we infer that $(\nabla F(w_N))_N$ is bounded in $\L^2(Q_T)$, where $F:\R_+\rightarrow\R_+$ is defined by 
\begin{align}
\label{eq:F}F(z) := \int_0^z \min(1,f_i(s)) \dd s.
\end{align}
$F$ is a strictly increasing $1$-Lipschitz function vanishing at $0$. In particular, we infer from the bound $(w_N)_N$ in $\L^2(Q_T)$ the same one for $(F(w_N))_N$. Up to a subsequence we can thus assume that $(w_N)_N$ and $(F(w_N))_N$ converge weakly to $w$ and $\widetilde{w}$ in $\L^2(Q_T)$. On the other hand, because of the equation satisfied by $\underline{u}_i^N$ and thanks to estimate \eqref{est:comp2} - \eqref{est:comp3}, one checks that $(\partial_t w_N)_N$ is bounded in $\mathscr{M}^1([0,T],\H^{-m}(\Omega))$, for $m\in\N$ large enough. We thus infer from Proposition 3 of \cite{mou} that (up to a subsequence), for any $\varphi\in\mathscr{C}^0(\overline{Q_T})$,
\begin{align}
\label{conv:comp}\int_{Q_T}  w_N F(w_N)\varphi \operatorname*{\longrightarrow}_{N+\infty} \int_{Q_T}  w \widetilde{w}\varphi.
\end{align}
At this stage we could invoke directly Proposition 1.4 [ACM], but since the framework is a little bit more general, for the reader convenience we reproduce here (a part of) the proof, which is in fact another occurence of the standard Minty-Browder or Leray-Lions trick (for the historical proof(s) see \cite{LL} and the references therein, or \cite{hunger} for a modern proof involving Young measures). One first establish that 
\begin{align*}
\int_{Q_T} \stackrel{:=h_N}{\overbrace{(F(w_N)-F(w))(w_N-w)}} &= \int_{Q_T} F(w_N)w_N+\int_{Q_T} F(w)w - \int_{Q_T}F(w_N)w-\int_{Q_T}F(w)w_N \\
& \operatorname*{\longrightarrow}_{N+\infty} \int_{Q_T} (w\tilde w -F(w)w+\tilde w w -F(w)w=0,
\end{align*}
by exploiting the $\L^2(Q_T)$ weak convergences $(w_N)_N \rightharpoonup_N w$, $(F(w_N))_N \rightharpoonup_N \widetilde{w}$, together with \eqref{conv:comp} when $\ffi=1$. Then, since $F$ is increasing, we have $h_N\geq 0$ so that the previous convergence may be seen as the convergence of $(h_N)_N$ to $0$ in $\L^1(Q_T)$. In particular, up to some subsequence, we get that that $(h_N)_N$ converges almost everywhere. Since $F$ is strictly increasing, for any sequence $(z_n)_n\in\R_+$ and $z\in\R_+$, the convergence of $(F(z_n)-F(z))(z_n-z)$ to $0$ implies that $(z_n)_n$ converges to $z$ and we recover in this way the fact that $(w_N)_N$ converges to $w$ a.e. on $Q_T$.

\subsection{Passing to the limit}
We already have the convergence of $\underline{U}^N$ to some $U$ almost everywhere. Therefore by continuity of $A$ and $R$, we have for the nonlinearities 
\begin{align*}
(\underline{U}^N, A(\underline{U}^N),R(\underline{U}^N),\mathcal{H}(\underline{U}^N))\operatorname*{\longrightarrow}_{N\rightarrow +\infty}^{\text{a.e.}} (U,A(U),R(U),\mathcal{H}(U)),
\end{align*}
so that, as said in the remark~\ref{rem:hyp}, passing to the limit is now just a question of (equi-)integrability for each of these three sequences. The first sequence $(\underline{U}^N)_N$ is bounded in $\L^2(Q_T)$ thanks to \eqref{est:comp2}, so that equi-integrability is automatic. For the second sequence $(A(\underline{U}^N))_N$, we use once more estimate \eqref{est:comp2}: since $A$ is continuous on $\R_+^I$, we infer that, for any norm $\|\cdot\|$ on $\R^I$,
\begin{align*}
\rho(R):=\max_{\|X\|\leq R}\|A(X)\|,
\end{align*} 
is a well-defined nondecreasing function of $R\geq 0$. Obviously, $\|A(X)\|>\rho(R)$ implies $\|X\|> R$, so that 
\begin{align*}
\int_{Q_T} \|A(\underline{U}^N)\|\mathbf{1}_{\|A(\underline{U}^N)\|>\rho(R)} &\leq \int_{Q_T} \|A(\underline{U}^N)\| \mathbf{1}_{\|\underline{U}^N\|>R}  \\
&\leq \frac{1}{R}\int_{Q_T} \|A(\underline{U}^N)\| \, \|\underline{U}^N\|,
\end{align*}
which goes to $0$ with $1/R$ uniformly in $N$, using \eqref{est:comp2} and the equivalence of norms in $\R^I$. As for the reaction terms, assumption \eqref{eq:equiint} allow to obtain the equi-integrability of $(R(\underline{U}^N))_N$ using the boudedness in $\L^1(Q_T)$ of $(\mathcal{H}(\underline{U}^N))_N$ and $(A(\underline{U}^N):\underline{U}^N)_N$. Eventually, using \eqref{est:comp3bis} and  \eqref{eq:gradient_control_dis}, we get for all $N$ 
\begin{align*}
  \int_\Omega\mathcal{H}(\underline{U}^N(t))  + \int_{0}^t\int_\Omega \langle \nabla \underline{U}^N, \D^2(\mathcal{H})(\underline{U}^N)\D(A)(\underline{U}^N)\nabla \underline{U}^N\rangle  \leq (1+e^{2CT})\left(CT+\int_\Omega \mathcal{H}(U_{\textnormal{in}})\right), 
\end{align*}
where we recall $\D^2(\mathcal{H})(\underline{U}^N)\D(A)(\underline{U}^N)\geq 0$, so that we get entropy estimate \eqref{eq:gradient_control} by the usual weak lower-semi continuity argument, ending the proof of Theorem \ref{thm:main}. \qed

\section{Application to separate variable cases. }\label{sec:structure}

\subsection{Examples : generalizing the use of detailed balance structure}\label{subsec:examp}
In  \cite{chen2016global},  Chen and coauthors identified the role of the detailed balance hypothesis for entropy structure of cross-diffusion systems in which the functions $p_i$ take the following form 
\begin{align*}
p_i(U)=d_i+\sum_{j=1}^I m_{ij}u_j^s.
\end{align*}
where $m_{ij}$ are nonnegative coefficients. Depending on the structure of the matrix, an entropy of the form 
$$
\mathcal{H}(U)=\sum_{i=1}^I \pi_i u_i^s,
$$
can be a uniform entropy if the the positive coefficients $\pi_i$ are adequately chosen. More precisely, it's the case if the following condition (introduced in \cite{chen2016global}) is fullfilled. 
\begin{definition}
We say that the matrix $M=(m_{ij})_{ij} \in\textnormal{M}_n(\R_+)$ satisfies the \emph{detailed balance} condition if
\begin{align}
\label{eq:db}\exists \pi\in(\R_+^*)^I\quad : \quad \forall i,j\quad  \pi_i m_{ij}=\pi_j m_{ji}.
\end{align}
\end{definition}
If $s\leq 1$, one can check that $\mathcal{H}(U)=\sum_{i=1}^I \pi_i u_i^s$ is indeed a uniform entropy for the system (with $f_i^2=sd_i u_i^{s-2}$). Let's try to generalize this procedure for other nonlinearities. For instance consider the case when the functions $p_i$ take the following form 
\begin{align*}
p_i(U)=d_i+\sum_{j=1}^I m_{ij} q_j(u_j),
\end{align*}
where $q_i\in \mathscr{C}^0(\R_+)\cap \mathscr{C}^1(\R_+^*)$ and $q_i'>0$.
This includes for instance multi-exponents cases that are $q_j(x) = x^{s_j}$. It seems unthinkable that the sole detailed balance leads to an entropy in this general framework: even when $I=2$ (that is, system with two populations), global weak solutions are (up to now) only known to exists (for the power-law case) under the condition that $s_1s_2 \leq 1$ (see \cite{dlmt}). It is therefore reasonnable to expect an extra condition if one wants to produce global weak solutions. We thus introduce the \emph{pairwise compatibility} which writes 
\begin{definition}[pairwise compatibility]
We say that the functions $q_j$ are pairwise compatible if
\begin{equation}\label{eq:pairwise}
\forall i\not=j,\quad \forall x,y>0,\quad q_i(x)q_j(y)-xyq_i'(x)q_j'(y)\geq 0. 
\end{equation}
\end{definition}

\vspace{2mm} 

This condition is reminiscent of the study performed in \cite{DesLepMou} which focused on  systems involving only two populations. We recall that in this latter framework (two populations with cross-diffusion coefficients $m_{12}q_2(u_2)$ and $m_{21}q_1(u_1)$), the entropy took the following form 
\begin{align*}
\ffi_1(u_1)+ \ffi_2(u_2),
\end{align*}  
where $\ffi_i$ is the only $\mathscr{C}^2(\R_+^*)$ function satisyfing $\ffi_i(1)= \ffi_i'(1)=0$ and
\begin{align*}
\ffi_i''(z) = \frac{q_i'(z)}{z}.
\end{align*}
Following \cite{chen2016global} a natural candidate for the entropy in our general setting would therefore be 
\begin{align}\label{eq:entpair}
\mathcal{H}(U) := \sum_{i=1}^I \pi_i \ffi_i(u_i).
\end{align}
This entropy is linked with our Theorem \ref{thm:main} through the following Lemma 
\begin{lemma}
Consider a family of nonnegative functions $(q_i)_{1\leq i\leq I}\in\mathscr{C}^0(\R_+)\cap\mathscr{C}^1(\R_+^*)$ such that $q_i'>0$. Then $\mathcal{H}$ defined by \eqref{eq:entpair} satisfies \eqref{ineq:H2}. Furthermore, if \eqref{eq:db}  and \eqref{eq:pairwise} hold, then $\mathcal{H}$ is a uniform entropy in the sense of Definition \ref{def:ent}.
\end{lemma} 
\begin{proof}
First, \eqref{ineq:H2} follows from elementary convex analysis, see point $(iv)$ of Lemma 3.5 in \cite{DesLepMou} for instance. The core of the matter is of course to check that $\mathcal{H}$ is a uniform entropy. For this purpose, for $X=(x_i)_i>0$ and $Z=(z_i)_i \in\R^I$ we compute
\begin{multline}
\langle Z, \D^2(\mathcal{H})(X)\D(A)(X)Z\rangle = \\
\label{ent:compu} \sum_{i=1}^I \pi_i\frac{q_i'(x_i)}{x_i} (d_i+m_{ii}q_i(x_i)+x_i q_i'(x_i)) z_i^2+\frac{1}{2}\sum_{i\neq j} \begin{pmatrix}z_i & z_j\end{pmatrix} B^{(i,j)}\begin{pmatrix}z_i \\ z_j\end{pmatrix},
\end{multline}
where the matrix $B^{(i,j)}$ is defined by 
\begin{align*}
B^{(i,j)}=\begin{pmatrix} \pi_i m_{ij} q_j(x_j)\frac{q_i'(x_i)}{x_i} &\pi_im_{ij}q_i'(x_i)q_j'(x_j)\\ 
\pi_jm_{ji}q_i'(x_i)q_j'(x_j)&  \pi_j m_{ji} q_i(x_i)\frac{q_j'(x_j)}{x_j}\end{pmatrix}.
\end{align*}
Due to the detailed balance condition \eqref{eq:db}, $B^{(i,j)}$ is a symmetric matrix. Since $q_i(u_i),q_i'(u_i)\geq 0$ it has nonnegative trace. Finally by the pairwise compatibility condition \eqref{eq:pairwise},
$$
\det(B^{(i,j)})=\pi_i\pi_jm_{ji}m_{ij}\frac{q_i'(u_i)q_j'(u_j)}{u_iu_j}\left( q_i(u_i)q_j(u_j)-u_iu_jq_i'(u_i)q_j'(u_j)\right)\geq 0.
$$  
Therefore $B^{(i,j)}$ are symmetric nonnegative matrices. Using again that $q_i,q_i'\geq 0$, \eqref{ent:compu} leads to 
\begin{align*}
 \langle Z, \D^2(\mathcal{H})(X)\D(A)(X)Z\rangle \geq \sum_{i=1}^I d_i \frac{q_i'(x_i)}{x_i} z_i^2,
\end{align*}
that is, in the sense of symmetric matrices,
\begin{align*}
 \D^2(\mathcal{H})(X)\D(A)(X) \geq \textnormal{Diag}(f_i(x_i))^2,
\end{align*}
where $f_i$ is simply the square-root of $x_i\mapsto d_i q_i'(x_i)/x_i$. 
\qed\end{proof}
%
%

\begin{remark}
This structure covers typically the case $q_i(u_i)=u_i^s$ covered in \cite{jungel2016cross}, in the case $s\leq 1$, the case $s>1$ is covered afterwards. It also covers the case $q_i(u_i)=u_i^{s_i}$ with all the $0<s_i\leq 1$ or $s_i>0$ and $s_is_j\leq 1$ for all $i\not=j$. This is the object of theorem~\ref{thm:genjun}.
\end{remark}

\textbf{Other cases} A lot of other cases can in fact be handled. Two typical generalization can be obtained (we focus on power law coefficients in order to clarify the message): 
\begin{itemize}
\item \textbf{relaxed detailed balance and $s_is_j<1$.} In case, $s_is_j<1$, an entropy structure can be found even in absence of detailed balance condition. Condition for matrix $B^{(i,j)}$ to be semi definite can then write as 
$$
\pi_i m_{ij}\pi_jm_{ji}s_is_j-s_i^2s_j^2\left(\frac{\pi_i m_{ij}+\pi_jm_{ji}}{2}\right)^2\geq 0.
$$
This is a (not very much) lighter constraint on the matrix. 
\item \textbf{self-diffusion $>$ cross-diffusion.} In case the hypothesis $s_is_j\leq 1$ is not satisfied or if neither detailed balanced or relaxed detailed balance conditions is satisfied, the entropic structure might hold because of additional self-diffusion. Indeed, in this case, the matrix $\D(A)(U)$ can exhibit eigenvalues with negative real part which might exclude any convex entropy (it can be checked easily on 2 species system). However if $m_{ii}$ is large enough, then the potentially problematic cross term in the quadratic form \eqref{ent:compu} adapted to power type coefficients.  They are  of the  form
$$
u_i^{s_i-1}u_j^{s_j-1} z_iz_j,
$$
For large $m_{ii}$'s the positive contribution of the form $u_i^{2s_i-2}z_i^2+u_j^{2s_j-2}z_j^2$ will control these terms. We only give here a general idea of compensation by cross-diffusion, for a deeper exploration of these possibilities, we refer to \cite{chen2016global}.
\end{itemize} 
With regards to \cite{chen2016global}, the main contribution of our approach is twofolds:
\begin{itemize}
\item firstly, we avoid any condition involving the dimension, the assumptions are the same, regardless of the dimension. 
\item secondly we show that a more general structure can be identified inspired from the one introduced in \cite{chen2016global}.
\end{itemize}

\subsection{Separate variables entropy: checking \textbf{H3}}\label{sec:Hadamard}

In this section, we identify a class of models where Theorem \ref{thm:main} may apply. They include in particular the examples of Subsection~\ref{subsec:examp}. All is based on the separation of variables. Note that the very definition of uniform entropy includes somehow a separation of variables.  In all what follows, besides hypothesis of theorem~\ref{thm:main}, we will make the following additional assumptions on the functions $p_i$ and the uniform entropy $\mathcal{H}$:
\begin{align}
\label{eq:pi_separe}
p_i(X)&=d_i+\sum_{j=1}^I q_{ij}(x_j),\quad q_{ij}\geq 0,\quad q_{ij} \in \mathscr{C}^0(\R_+)\cap \mathscr{C}^1(\R_+^*),\\
\label{eq:H_separe}
\mathcal{H}(X)&=\sum_{i=1}^I \mathcal{H}_i (x_i) 
\end{align}
which means that both cross-diffusion pressures and entropy can be written in separate variables form. 
We focus on this case for one reason: as discussed in remark~\ref{rem:hyp}, only hypothesis \textbf{H3} seems restrictive regarding the structure we are focused on and it gives a quite general framework, larger than $\mathscr{C}^1(\R_+^I)$ (which was discussed in \cite{dlmt}) for which this hypothesis is essentially a straighforward consequence of the entropy structure. 
Assumption \eqref{eq:H_separe} and the fact that $\mathcal{H}$ is a uniform entropy imply in particular for all $X > 0$
\begin{align*}
\textnormal{Diag}(\mathcal{H}_i''(x_i)) \D(A)(X)\geq \textnormal{Diag}(f_i(x_i)^2),
\end{align*}
where $f_i\in\mathscr{C}^0(\R_+^*,\R_+^*)$. To check \textbf{H3}, we first prove the partial result 
\begin{proposition}\label{prop:ouvert}
$A:(\R_+^*)^I \rightarrow (\R_+^*)^I$ is a $\mathscr{C}^1$-diffeomorphism.
\end{proposition}
\begin{proof}
First, since $q_{ij}\geq 0$ and $d_i>0$, $A$ maps $(\R_+^*)^I$ to itself. Now, consider the map $\Phi:\R^I\rightarrow \R^I$ defined by 
\begin{align}
\label{eq:phi}\Phi(X) := \ln(A(\exp(X)),
\end{align}
where the functions $\ln$ and $\exp$ have both to be understood coordinate by coordinate when applied to a vector. We will invoke the Hadamard-Lévy Theorem, that we recall for the reader's convenience:
\begin{theorem}[Hadamard-L\'evy]
A $\mathscr{C}^1$ map $\Phi:\R^d\rightarrow\R^d$ is a $\mathscr{C}^1$-diffeormopshim if and only it is a proper map without critical points. 
\end{theorem}
Let us check that  $\Phi$ defined by \eqref{eq:phi} is proper. If $(\Phi(X^n))_n$ is bounded for some sequence $(X^n)_n \in\R^I$, this implies the existence of $0<\beta_1<\beta_2$ such that 
\begin{align*}
 \forall i=1,\dots I, \quad \beta_1 \leq A_i(\exp(X^n)) \leq \beta_2.
\end{align*}
 Using the hypothesis \textbf{H2}, we have the existence of $\alpha>0$, such that $A_i(\exp(X^n)) \geq \alpha \exp(x^n_i)$. In particular 
$$
0< \exp(x_i^n)\leq \frac{\beta_2}{\alpha}=:M.
$$
As an immediate consequence, we have 
$$
\beta_1\leq A_i(\exp(X^n))=p_i(\exp(X^n)) e^{x_i^n}\leq e^{x_i^n} \stackrel{:=C_i>0}{\overbrace{\sup_{X\in[0,M]^I} p_i(X)}}
$$
Leading immediately to
$$
\log\left(\frac{\beta_1}{C_i}\right)\leq x_i^n\leq M, 
$$
which ensures the boundedness of $X^n$ and thereby that $A$ is proper. Since the $p_i$ functions are all $\mathscr{C}^1$ on $(\R_+^*)^I$, so is $A$ on $(\R_+^*)^I$ and it has no critical points on this domain since we have by entropy assumption for $X>0$
\begin{align*}
\D^2(\mathcal{H})(X)\D(A)(U) \geq \textnormal{Diag}(f_i(x_i))^2 > 0.
 \end{align*}   Since both $\ln:(\R_+^*)^I \rightarrow \R^I$ and $\exp:\R^I \rightarrow (\R_+^*)^I$ are $\mathscr{C}^1$ without critical points, the chain rule leads to the same conclusion for $\Phi:\R^I \rightarrow \R^I$ and the Hadamard-L\'evy Theorem allows to conclude. \qed
\end{proof} 
\begin{remark}
The reader might notice that this part does not use the separability of variables. It is only essential is the sequel. 
\end{remark}
\vspace{2mm}

The remaining difficulty consists in the treatment of the boundary of $\R_+^I$. The idea consists in splitting it into the union of $\{0,\cdots 0\}^I$ and sets of the form $\Sigma:=V_1\times \cdots\times V_I$ where for all $V_i\in\{\{0\},\R_+^*\}$ (that is certain coordinates are frozen to $0$ and the other are positive). It is straightforward that $A^{-1}(X)\rightarrow 0$ as $X\rightarrow 0$ and that $A$ preserves any such set $\Sigma$. 
We prove here that $A$ induces an homeomorphism on such sets $\Sigma$. Since the idea of the proof is the same for any such set, we just establish it for $\Sigma=(\R_+^*)^{I-1}\times \{0\}$.  
The map $A:\R_+^I\rightarrow \R_+^I$ induces a map $\widetilde{A} : \R_+^{I-1} \rightarrow \R_+^{I-1} $ by the following formula for $\widetilde{X}:=(x_1,\cdots,x_{I-1}) \in \R_+^{I-1}$ and $1\leq i\leq I-1$
\begin{align*}
\widetilde{A}(\widetilde{X})_i := A(x_1,\cdots,x_{I-1}, 0)_i.
\end{align*}
The proper character of $\widetilde{A}$ on $\R_+^{I-1}$ is inherited from $A$. The map $\widetilde{A}$ is also differentiable (it is $\mathscr{C}^1$) on $\Sigma$ and has no singular points due to the entropy structure. Indeed, seeing $\widetilde{A}$ as a map on $(\R_+^*)^{I-1}$, we have 
$$
\begin{cases}
\D(\widetilde{A})(X)_{ii}=d_i+\sum_{j=1}^{I-1} a_{ij}q_j(x_j)+a_{iI}q_I(0)+a_{ii}x_iq_i'(x_i),\\
\D(\widetilde{A})(X)_{ij}=x_iq_j'(x_j),
\end{cases}
$$
so that 
$$
\D^2(\widetilde{\mathcal{H}})\D(\widetilde{A})>0,
$$
which ensures that $\widetilde{A}$  satisfies the hypothesis of the Hadamard Levy theorem. We thus recover that  $A$ is one to one on $\R_+^I$. This ensures that $A^{-1}$ is well defined on $\R_+^I$. We just need to prove its continuity. Let $X^n\rightarrow X$, then if $X\in(\R_+^*)^I$, we already know that $A^{-1}(X^n)\rightarrow A^{-1}(X)$. Otherwise, since $A$ is proper, we know that $A^{-1}(X^n)$ is bounded. We extract a convergent subsequence such that $A^{-1}(X^n)\rightarrow Y$ and since $A$ is one to one and continuous, the only possibility is $A(Y)=A(X)$ meaning that $A^{-1}(X^n)$ converges to $A^{-1}(X)$ which ends the proof.
 
\begin{remark}
The (small) generalization w.r.t. \cite{dlmt} concerns the third assumption which can be verified more easily without assuming $A$ is $\mathscr{C}^1$ up to the boundary. 
\end{remark}
\subsection{Proof of Theorem~\ref{thm:genjun}}
From the previous subsections we already know that hypothesis \textbf{H1}, \textbf{H2}, \textbf{H3} are fullfilled, and that $\mathcal{H}$ is a uniform entropy. We only need to check \eqref{ineq:H2} and \eqref{eq:equiint}. For the first one, we write 
$$
\nabla \mathcal{H}(X)\cdot R(X) =\sum_{i=1}^I \pi_i h_i'(x_i)x_i\left(r_i-\sum_{j=1}^I c_{ij}x_j^{\alpha_{ij}}\right)
$$
Before entering the computations, we recall a few properties on the functions $h_i$ associated to power type coefficients. Up to a multiplication by positive constants, we have for any $z\in\R_+$

$$
h_i(z)=\begin{cases}z\log z -z+ 1\text{ if }s_i=1,\\
 \dfrac{z^{s_i}-s_iz-(s_i-1)}{s_i(s_i-1)}\text{ otherwise. }\end{cases}
$$ 
It is easy to see that for some constant $C_i$ we have 
$$
-C_i\leq zh_i'(z)\leq C_i(h_i(z)+1)\qquad \text{ and }\quad u_i\leq C_i(1+h_i(u_i)).
$$
Therefore, we have immediately 
\begin{align*}
h_i'(x_i)x_i\left(r_i-\sum_{j=1}^I c_{ij}x_j^{\alpha_{ij}}\right)&\leq C_i\sum_{j=1}^I c_{ij}x_j^{\alpha_{ij}}+C_ir_i(h_i(x_i)+1).
\end{align*}
Using the hypothesis $\alpha_{ij}<1$, we have $x_i^{\alpha_{ij}}\leq 1+x_j$, leading to 
$$
\sum_{i=1}^I\pi_i h_i'(x_i)\left(r_i-\sum_j c_{ij}x_j^{\alpha_{ij}}\right)\leq C\left(1+\mathcal{H}(X)\right).
$$
This ends the verification of \eqref{ineq:H2}.
The hypothesis \eqref{eq:equiint} is just an immediate consequence of the fact that $\alpha_{ij}<1$. We can thus invoke Theorem~\ref{thm:main}, ending the proof of Theorem~\ref{thm:genjun} \qed
\section{Conclusion}
Apart from the generalization of the entropic structure from \cite{chen2016global}, the main message of this paper is that the approximation procedure designed in \cite{dlmt} is very robust as soon as system can be written in the form \eqref{eq:cross}.  The most technical points in practise are 
\begin{itemize}
\item checking that $U\mapsto (p_i(U)u_i)_i$ defines an homeomorphism on $\R_+^I$ (especially the treatment of the boundary),
\item checking that the gradient control arising from the entropy dissipation allows the extraction of a almost everywhere converging sequence. 
\end{itemize}
The approximation procedure preserves the main estimates of interests of the equation. Next step could be the treatment of potentially flat coefficients for which the entropy dissipation does not give directly a sufficient gradient estimate. A new approach will be needed in this case (possibly a perturbative approach) since the invertibility of $A$ is quite crucial in our process. 

\section*{Appendix}
We prove here Lemma \ref{lem:dualsum} which emphasizes the necessity of the assumption $p_i\in\mathscr{C}^0(\R_+^I)$ in our approach. In particular, this is why some examples are still resisting to our procedure, see \cite{Zamponi} for instance, and why it is not only a technical condition.  The proof can be seen as a bootstrap of the classical dual estimate and was actually already noticed in \cite{Lepoutre_JMPA}, Section 3, for the conservative case.

\vspace{2mm}

\emph{Proof of Lemma \ref{lem:dualsum}}

\vspace{2mm}

 Summing up the equations, we have, denoting $v=\sum_{i=1}^I u_i$ 
\begin{align*}
\partial_t v-\Delta \big[\mu v\big]\leq \rho v,
\end{align*}
where 
\begin{align*}
 \mu:=\frac{\sum_{i=1}^I p_i(U) u_i}{\sum_{i=1}^I u_i}.
\end{align*}
We thus infer  from lemma~\ref{lem:dualite}
\begin{align}\label{ineq:proof:dualsum}
\int_{Q_T} \mu v^2\leq C\left(1+\int_{Q_T} \mu\right),
\end{align}
where the constant $C$ depends on $\Omega,T,\rho$ and the initial data. For $R>0$, the continuity on $\R_+^I$ of the $p_i$ functions allow us to define
\begin{align*}
M(R)=\sup_{\sum_{i=1}^I x_i\leq R} \frac{\sum_{i=1}^I p_i(X)x_i}{\sum_{i=1}^I x_i},
\end{align*}
where the supremum is taken over all $X=(x_i)_{1\leq i \leq I})\in\R_+^I$. Using this function we write 
\begin{align*}
  \int_{Q_T} \mu = \int_{Q_T} \mu\mathbf{1}_{v>R} + \int_{Q_T} \mu\mathbf{1}_{v\leq R} \leq \frac{1}{R^2} \int_{Q_T} \mu v^2 + M(R)|Q_T|.
\end{align*}
Using the previous inequality for $R$ large enough in \eqref{ineq:proof:dualsum}, we get \eqref{ineq:dualsum1}. 
\qed

\medskip

{\bf{Acknowledgement}}: 
The research leading to this paper was funded by the french "ANR blanche" project Kibord: ANR-13-BS01-0004.
\medskip

\bibliographystyle{alpha}
\bibliography{LM} 
\end{document}